\documentclass[a4paper,twoside,12pt]{article}
\usepackage{amssymb,amsmath,amsthm,latexsym}
\usepackage{amsfonts}
\usepackage{graphicx,caption,subcaption}
\usepackage[pdftex,bookmarks,colorlinks=false]{hyperref}
\usepackage[hmargin=1.2in,vmargin=1.2in]{geometry}
\usepackage[mathscr]{euscript}

\usepackage[auth-sc-lg,affil-sl]{authblk}
\usepackage{setspace}

\def\ni{\noindent}

\def\S {\Sigma}
\def\s {\sigma}
\def\cP{\mathcal{P}}
\def\f{f^{\oplus}}

\newtheorem{thm}{Theorem}[section]

\newtheorem{cor}[thm]{Corollary}
\newtheorem{defn}[thm]{Definition}

\newtheorem{prob}{Problem}

\title{\textbf{\sc A Study on Set-Valuations of Signed Graphs}}

\author{P. K. Ashraf}
\affil{\small Department of Mathematics\\ Government Arts and Science College \\ Koduvally, Kozhikkode - 676xxx, Kerala, India.\\ E-mail: ashrafkalanthod@gmail.com}

\author{K. A. Germina}
\affil{\small Department of Mathematics\\ University of Botswana\\ Gaborone, Botswana.\\ E-mail: srgerminaka@gmail.com}

\author{N. K. Sudev\footnote{Corresponding author}}
\affil{\small Department of Mathematics\\ Vidya Academy of Science \& Technology \\ Thalakkottukara, Thrissur - 680501, Kerala, India.\\ E-mail: sudevnk@gmail.com}

\date{}

\begin{document}
\maketitle

\begin{abstract}
Let $X$ be a non-empty ground set and $\cP(X)$ be its power set. A set-labeling (or a set-valuation) of a graph $G$ is an injective set-valued function $f:V(G)\to \cP(X)$ such that the induced function $\f:E(G) \to \cP(X)$ is defined by $\f(uv)=f(u)\oplus f(v)$, where $f(u)\oplus f(v)$ is the symmetric difference of the sets $f(u)$ and $f(v)$. A graph which admits a set-labeling is known to be a  set-labeled graph. A set-labeling $f$ of a graph $G$ is said to be a set-indexer of $G$ if the associated function $\f$ is also injective.  In this paper, we define the notion of set-valuations of signed graphs and discuss certain properties of signed graphs which admits certain types of set-valuations.
\end{abstract}

\vspace{0.2cm}

\ni \textbf{Key words}: Signed graphs; balanced signed graphs; clustering of signed graphs; set-labeled signed graphs.

\vspace{0.04in}

\ni \textbf{AMS Subject Classification} : 05C78, 05C22. 

\section{Introduction}

For all  terms and definitions, not defined specifically in this paper, we refer to \cite{BM1,FH,DBW} and  and for the topics in signed graphs we refer to \cite{TZ1,TZ2}. Unless mentioned otherwise, all graphs considered here are simple, finite, undirected and have no isolated vertices.

\subsection{An Overview of Set-Valued Graphs}

Let $X$ be a non-empty set and $\cP(X)$ be its power set. A {\em set-labeling} (or a \textit{set-valuation}) of a graph $G$ is an injective function $f:V(G)\to \cP(X)$ such that the induced function $\f:E(G)\to \cP(X)$ is defined by $\f(uv)=f(u)\oplus f(v)~ \forall ~ uv\in E(G)$, where $\oplus$ is the symmetric difference of two sets.  A graph $G$ which admits a set-labeling is called an {\em set-labeled graph} (or a set-valued graph)(see \cite{BDA1}).  

A {\em set-indexer} of a graph $G$ is an injective function $f:V(G)\to \cP(X)$ such that the induced function $\f:E(G) \to \cP(X)$ is also injective. A graph $G$ which admits a set-indexer is called a {\em set-indexed graph} (see \cite{BDA1}).

Several types of set-valuations of graphs have been introduced in later studies and their properties and structural characteristics of such set-valued graphs have been done extensively.

\subsection{Preliminaries on Signed Graphs}

An edge of a graph $G$ having only one end vertex is known as a \textit{half edge} of $G$ and an edge of $G$ without end vertices is called \textit{loose edge} of $G$. 

A \textit{signed graph} (see \cite{TZ1,TZ2}), denoted by $\S(G,\s)$,  is a graph $G(V,E)$ together with a function $\s:E(G)\to \{+,-\}$ that assigns a sign, either $+$ or $-$, to each ordinary edge in $G$. The function $\s$ is called the {\em signature} or {\em sign function} of $\S$, which is defined on all edges except half edges and is required to be positive on free loops.

An edge $e$ of a signed graph $\S$ is said to be a \textit{positive edge} if $\s(e)=+$ and an edge $\s(e)$ of a signed graph $\S$ is said to be a \textit{negative edge} if $\s(e)=-$. The set $E^+$ denotes the set of all positive edges in $\S$ and the set $E^-$ denotes the set of negative edges in $\S$. 

A  simple  cycle (or path) of a signed graph $\S$  is said to be {\em balanced} (see \cite{AACE,FHS}) if the product of signs of its edges is $+$. A  signed  graph $\S$ is said to be a {\em balanced signed graph} if it contains no half edges and all of its simple cycles are balanced.  It is to be noted that the number of all negative  signed graph is balanced if and only if it is bipartite. 

Balance or imbalance is the basic and the most important property of a signed graph. The following theorem, popularly known as {\em Harary's Balance Theorem}, establishes a criteria for balance in a signed graph.

\begin{thm}
	{\rm \cite{FHS}} The following statements about a signed graph are equivalent.
	\begin{enumerate}\itemsep0mm
	\item[(i)] A signed graph $\S$ is balanced.
	\item[(ii)] $\S$ has no half edges and there is a partition $(V_1,V_2)$ of $V(\S)$ such that $E^-=E(V_1,V_2)$.
	\item[(iii)] $\S$ has no half edges and any two paths with the same end points have the same sign. 
\end{enumerate}
\end{thm}

Some balancing properties of certain types of signed graphs have been studied in \cite{GSH, GSH2}.

A signed graph $\S$ is said to be  \textit{clusterable} or \textit{partitionable} (see \cite{TZ1,TZ2}) if its vertex set can be partitioned into subsets, called \textit{clusters}, so that every positive edge joins the vertices within the same cluster and every negative edge joins the vertices in the different clusters. If $V(\S)$ can be partitioned in to $k$ subsets with the above mentioned conditions, then the signed graph $\S$ is said to be \textit{$k$-clusterable}. In this paper, we study the $2$-clusterability of signed graphs only.

Note that $2$-clusterability always implies balance in a signed graph $\S$. But, the converse need not be true. If all edges in $\S$ are positive edges, then $\S$ is balanced but not $2$-clusterable.

In this paper, we introduce the notion of set-valuations of signed graphs and study the properties and characteristics of such signed graphs.

\section{Set-Labeled Signed Graphs}

Motivated from the studies on set-valuations of signed digraphs in \cite{BDAS}, and the studies on integer additive set-labeled signed graphs in \cite{GS5}, we define the notion of a set-labeling of a signed graph as follows.

\begin{defn}{\rm
	Let $X$ be a non-empty set and let $\S$ be a signed graph, with corresponding underlying graph $G$ and the signature $\s$. An injective function $f:V(\S)\to \cP(X)$ is said to be a \textit{set-labeling} (or \textit{set-valuation}) of $\S$ if $f$ is a set-labeling of the underlying graph $G$ and the signature of $\S$ is defined by $\s(uv)=(-1)^{|f(u)\oplus f(v)|}$. A signed graph $\S$ together with a set-labeling $f$ is known as a \textit{set-labeled signed graph} (or set valued signed graph) and is denoted by $\S_f$. }
\end{defn}

\begin{defn}{\rm 
A set-labeling $f$ of a signed graph $\S$ is said to be a set-indexer of $\S$ if $f$ is a set-indexer of the underlying graph $G$.}
\end{defn}

If the context is clear, we can represent a set-valued signed graph or a set-indexed signed graph simply by $\S$ itself. In this section, we discuss the $2$-clusterability and balance of set-valued signed graphs.

The following theorem establishes the existence of set-valuations for all signed graphs.

\begin{thm}
	Every signed graph admits a set-labeling (and a set-indexer).
\end{thm}
\begin{proof}
	Let $\S$ be a signed graph whose vertex set is given by $V(\S)=\{v_1,v_2,\ldots, v_n\}$. Let $X=\{1,2,3,\ldots, n\}$. Define a set-valued function $f:V(\S)\to \cP(X)$ such that $f(v_i)=\{i\}$, where $1\le i\le n$. Clearly, $f$ is an injective function.  Then, $\f(v_iv_j)=\{i,j\}, \forall\ uv\in E(G)$. Note that $\f$ is also an injective function and hence $f$ is a set-indexer of $\S$.
\end{proof}

We say that two sets are of \textit{same parity} if they are simultaneously even or simultaneously odd. If two sets are not of same parity, then they are said to be the sets of \textit{opposite parity}. The signature of an edge of a set-valued signed graph can be determined in terms of the set-labels of its end vertices, as described in the following theorem.

\begin{thm}\label{Thm-2.1}
	An edge $e$ of a set-labeled signed graph is a positive edge if and only if the set-labels of its end vertices are of the same parity.
\end{thm}
\begin{proof}
	Let $f$ be a set-labeling of a given signed graph $\S$. Assume that, an edge $e=v_iv_j$ be a positive edge in $\S$. Then, $|f(v_i)\oplus f(v_j)|=|f(v_i)-f(v_j)|+|f(v_j)-f(v_i)|$ is an even number. That is, $|f(v_i)-f(v_j)|$ and $|f(v_j)-f(v_i)|$ are simultaneously even or simultaneously odd. Hence, we need to consider the following cases.
	
	{\em Case-1:} Assume that both $|f(v_i)-f(v_j)|$ and $|f(v_j)-f(v_i)|$ are even. That is, both $|f(v_i)-f(v_i)\cap f(v_j)|$ and $|f(v_j)-f(v_i)\cap f(v_j)|$ are even. Then, we have
	
	{\em Subcase-1.1:} Let $|f(v_i)|$ be an even integer. Then, since $|f(v_i)-f(v_i)\cap f(v_j)|=|f(v_i)|-|f(v_i)\cap f(v_j)|$, we have $|f(v_i)\cap f(v_j)|$ must also be even. Hence, as $|f(v_j)-f(v_i)\cap f(v_j)|=|f(v_j)|-|f(v_i)\cap f(v_j)|$ is even, we have $|f(v_j)|$ is even.
	
	{\em Subcase-1.2:} Let $|f(v_i)|$ be an odd integer. Then, since $|f(v_i)-f(v_i)\cap f(v_j)|=|f(v_i)|-|f(v_i)\cap f(v_j)|$ is even, we have $|f(v_i)\cap f(v_j)|$ must be odd. Hence, as $|f(v_j)-f(v_i)\cap f(v_j)|=|f(v_j)|-|f(v_i)\cap f(v_j)|$ is even, we have $|f(v_j)|$ is odd. 
	
	{\em Case-2:} Assume that both $|f(v_i)-f(v_j)|$ and $|f(v_j)-f(v_i)|$ are odd. That is, both $|f(v_i)-f(v_i)\cap f(v_j)|$ and $|f(v_j)-f(v_i)\cap f(v_j)|$ are odd. Then, we have
	
	{\em Subcase-2.1:} Let $|f(v_i)|$ be an even integer. Then, since $|f(v_i)-f(v_i)\cap f(v_j)|$ is odd, we have $|f(v_i)\cap f(v_j)|$ must be odd. Hence, as $|f(v_j)-f(v_i)\cap f(v_j)|$ is odd, we have $|f(v_j)|$ is even.
	
	{\em Subcase-2.2:} Let $|f(v_i)|$ be an odd integer. Then, since $|f(v_i)-f(v_i)\cap f(v_j)|$ is odd, we have $|f(v_i)\cap f(v_j)|$ must be even. Then, as $|f(v_j)-f(v_i)\cap f(v_j)|=|f(v_j)|-|f(v_i)\cap f(v_j)|$ is odd, we have $|f(v_j)|$ is odd.    
\end{proof}

\ni As a contrapositive of Theorem \ref{Thm-2.1}, we can prove the following theorem also.

\begin{thm}\label{Thm-2.2}
	An edge $e$ of a set-labeled signed graph is a negative edge if and only if the set-labels of its end vertices are of the opposite parity.
\end{thm}

\ni The following result is an immediate consequence of Theorem \ref{Thm-2.1} and \ref{Thm-2.2}.

\begin{cor}
	A set-valued signed graph $\S$ is balanced if and only if every cycle in $\S$ has even number of edges whose end vertices have opposite parity set-labels.
\end{cor}
\begin{proof}
	Note that the number of negative edges in any cycle of a balanced signed graph is even. Hence, the proof is immediate from Theorem \ref{Thm-2.2}.
\end{proof}

The following theorem discusses a necessary and sufficient condition for a set-valued signed graph to be $2$-clusterable.

\begin{thm}\label{Thm-2.3}
	A set-valued signed graph is $2$-clusterable if and only if at least two adjacent vertices in $\S$ have opposite parity set-labels.
\end{thm}
\begin{proof}
	First, assume that at least two adjacent vertices in the set-valued signed graph $\S$ have opposite parity set-labels. If $e=v_iv_j$ be an edge of $\S$ such that $f(v_i)$ and $f(v_j)$ are of opposite parity, then $\s(v_iv_j)=-$. Then, we can find $(U_1, U_2)$ be a partition of $V(\S)$ such that $U_1$ contains one end vertex of every negative edge and $U_2$ contains the other end vertex of every negative edge. Therefore, $\S$ is $2$-clusterable.
	
	Conversely, assume that $\S$ is $2$-clusterable. Then, there exist two non-empty subsets $U_1$ and $U_2$ of $V(\S)$ such that $U_1\cup U_2=V(\S)$. Since $\S$ is a connected signed graph, at least one vertex in $U_1$ is adjacent some vertices in $U_2$ and vice versa. Let $e=v_iv_j$ be such an edge in $\S$ Since $\S$ is $2$-clusterable, $e$ is a negative edge  and hence $f(v_i)$ and $f(v_j)$ are of opposite parity. This completes the proof. 
\end{proof}

\begin{thm}
	Let $f$ be a set-indexer defined on a signed graph $\S$ whose underlined graph $G$ is an Eulerian graph. If $\S$ is balanced, then $\sum\limits_{e\in E(\S)}|\f(e)|\equiv 0\ ({\rm mod}\ 2)$.
\end{thm}
\begin{proof}
	Let the underlying graph $G$ of $\S$ is Eulerian.  Then, $G=\bigcup\limits_{i=1}^k C_i$, where each $C_i$ is a cycle such that $C_i$ and $C_j$ are edge-disjoint for $i\ne j$. Let $E_i$ be the edge set of the cycle $C_i$. Since $f$ is a set-indexer of $\S$, we have $f(E_i)\cap f(E_j) =\emptyset$, for $i\ne j$. Hence, we have
	\begin{equation}
	\sum\limits_{e\in E(\S)}|f(e)|=\sum\limits_{i=1}^k\sum\limits_{e_i\in E_i}f(e_i) \label{eqn-1}
	\end{equation}
		
	Consider the cycle $C_i$. Let $A_i$ be the set all positive edges and $B_i$ be the set of all negative edges in the cycle $C_i$. Then, the set-labels of edges in $A_i$ are of even parity and those of edges in $B_i$ are of odd parity. That is, $|\f(e)|\equiv 0\ ({\rm mod}\ 2)$ for all $e\in A_i$ and hence we have
	\begin{equation}
	\sum\limits_{e\in A_i}|f(e)|\equiv 0\ ({\rm mod}\ 2) \label{eqn-2}
	\end{equation}
	
	Since $\S$ is balanced, the number of negative edges in $C_i$ is even. Therefore, the number of elements in $B_i$ must be even. That is, the number of edges having odd parity set-labels in $C_i$ is even. Therefore, being a sum of even number of odd integers, we have 
		\begin{equation}
		\sum\limits_{e_i\in B_i}|f(e_i)|\equiv 0\ ({\rm mod}\ 2) \label{eqn-3}
		\end{equation}
	
	\ni From Equation \eqref{eqn-2} and Equation \eqref{eqn-3}, we have 
	\begin{equation}
	\sum\limits_{e\in E_i}|f(e)|=\sum\limits_{e_i\in A_i}|f(e_i)|+\sum\limits_{e_i\in B_i}|f(e_i)|\equiv 0\ ({\rm mod}\ 2). \label{eqn-4}
	\end{equation}
	
	\ni Therefore, by Equation \eqref{eqn-1} and Equation \eqref{eqn-4}, we can conclude that $$\sum\limits_{e\in E(\S)}|\f(e)|\equiv 0\ ({\rm mod}\ 2).$$
	\vspace{-1cm}
	
\end{proof}

From the above results, we infer the most important result on a set-valued signed graph as follows. 

\begin{thm}\label{Thm-2.4}
	If a signed graph $\S$ admits a vertex set-labeling, then $\S$ is balanced.
\end{thm}
\begin{proof}
	Let $\S$ be a signed graph which admits a set-labeling. If all vertices of $\S$ have the same parity set-labels, then by Theorem \ref{Thm-2.1}, all edges of $\S$ are positive edges and hence $\S$ is balanced. 
	
	Next, let that $\S$ contains vertices with opposite parity set-labels. Let $A_i$ be the set of all vertices with odd parity set-labels and $B_i$ be the set of all even parity set-labels. First, assume $v_i$ be a vertex in $A_i$ whose adjacent vertices are in $B_i$. Then, $v_i$ is one end vertex of some negative edges in $\S$. If $v_i$ is not in a cycle of $\S$, then none of these negative edges will be a part in any cycle of $\S$. 
		
	If $v_i$ is an internal vertex of a cycle $C$, then it is adjacent to two vertices, say $v_j$ and $v_k$, which are in $B_i$. Hence, the edges $v_iv_j$ and $v_iv_k$ are negative edges. If two vertices $v_i$ and $v_j$ are adjacent in the cycle $C$, then $v_i$ is adjacent to one more vertex, say $v_k$ and the vertex $v_j$ is also adjacent to one more vertex $v_l$ and in the cycle $C$, the edges $v_iv_k$ and $v_jv_l$ are negative edges and the edge $v_iv_j$ is a positive edge. If the vertices $v_i$ and $v_j$ are not adjacent, then also each of them induce two negative, which may not be distinct always. However, in each case the number of negative edges will be even. This condition can be verified in all cases when any number element of $A_i$ are the vertices of any cycle $C$ in $\S$. Hence, the umber of negative edges in any cycle of a set-labeled signed graph is even. Hence, $\S$ is balanced.
	
	Hence, in this case, the number negative edges in $C$ will always be even. 
\end{proof}

It is interesting to check whether the converse of the above theorem is valid. In context of set-labeling of signed graphs, a necessary and sufficient condition for a signed graph $\S$ is to be balanced is given in the following theorem.

\begin{thm}\label{Thm-2.5}
	A signed graph $\S$ is balanced if and only if it admits a set-labeling.
\end{thm}
\begin{proof}
	The proof is an immediate consequence of Theorem \ref{Thm-2.1} and Theorem \ref{Thm-2.4}.
\end{proof}

In view of Theorem \ref{Thm-2.1} and Theorem \ref{Thm-2.5}, we have 

\begin{thm}
	Any set-labeled signed graph is balanced.
\end{thm}

\section{Conclusion}

In this paper, we have discussed the characteristics and properties of the signed graphs which admit set-labeling with a focus on $2$-clusterability and balance of these signed graphs. There are several open problems in this area. Some of the open problems that seem to be promising for further investigations are following.

\begin{prob}{\rm 
Discuss the $k$-clusterability of different types of set-labeled signed graphs for $k>2$.}
\end{prob}

\begin{prob}{\rm 
Discuss the balance, $2$-clusterability and general $k$-clusterability of other types of  set-labeling of signed graphs such as topological set-labeling, topogenic set-labeling, graceful set-labeling, sequential set-labeling etc.}
\end{prob}

\begin{prob}{\rm 
Discuss the balance and $2$-clusterability and general $k$-clusterability of different set-labeling of signed graphs, with different set operations other than the symmetric difference of sets.}
\end{prob}

Further studies on other characteristics of signed graphs corresponding to different set-labeled graphs are also interesting and challenging. All these facts highlight the scope for further studies in this area. 

\section*{Acknowledgement}

The authors would like to dedicate this work to (Late) Prof. (Dr.) Belamannu Devadas Acharya, who had been the motivator and the role model for them and who have introduced the concept of set-valuations of graphs.

\end{document}